\documentclass[a4paper,11pt]{amsart}
\usepackage{amssymb}
\usepackage{euscript,eufrak,verbatim}
\usepackage[dvips]{graphicx}
\usepackage{verbatim}
\usepackage[final]{epsfig}
\usepackage{ccaption}

\newtheorem{corollary}{Corollary}

\newtheorem{lemma}{Lemma}

\theoremstyle{definition}

\newtheorem{remark}{Remark}

\DeclareMathSymbol{\varnothing}{\mathord}{AMSb}{"3F}

 \title[wandering intervals and acim]
      {Wandering intervals and absolutely continuous invariant probability measures of interval maps}

\author[ Hongfei Cui and Yiming Ding]{}\thanks{{\it Mathematical classification (2000):} 37E05, 28D05, 37D25.}

\thanks{This work was partly supported from the NSFC grant 60534080.}

\begin{document}

\maketitle

\centerline{\scshape Hongfei Cui $^{a, b}$ and Yiming Ding $^{a}$}
\medskip
{\footnotesize
 \centerline{$^{a}$ Wuhan Institute of Physics and Mathematics}
  \centerline{The Chinese Academy of Sciences, Wuhan 430071, China}
} 

\medskip
\medskip
{\footnotesize \centerline{$^{b}$ Graduate School of the Chinese
Academy of Sciences} \centerline{The Chinese Academy of Sciences,
Beijing 10049, China}
\centerline{Email Address:
cuihongfei05@mails.gucas.ac.cn, ding@wipm.ac.cn}
} 

\medskip

\begin{abstract}
For piecewise $C^1$ interval maps possibly containing critical
points and discontinuities with negative Schwarzian derivative,
under two summability conditions on the growth of the derivative
and recurrence along critical orbits, we prove
\begin{enumerate}
\item the nonexistence of wandering intervals, \item the existence of
absolutely continuous invariant measures, and \item the bounded
backward contraction property.
\end{enumerate}
The proofs are based on the method of proving the existence of
absolutely continuous invariant measures of unimodal map, developed
by Nowicki and van Strien.

\end{abstract}

\section{Introduction and main results}

The concept of wandering intervals plays an important role in
studying dynamical behavior of non-uniformly hyperbolic dynamical
system. In the area of interval dynamics, most important results
are related to the absence of wandering intervals. Our main aim in
this paper is to obtain a condition on the orbits of the critical
values to ensure the absence of wandering intervals for piecewise
$C^1$ interval maps with critical points and discontinuities, and
to give a sufficient condition to the nonexistence of wandering
intervals for multimodal maps whose orders of critical points are
different to the left and to the right.

The  motivation to show nonexistence of wandering intervals are well
known. Firstly, it is relevant to the isomorphism problem of
dynamical system. In the 1880s, Poincar$\acute{e}$  proved that each
orientation preserving homeomorphism of the circle without periodic
points is semi-conjugate to an irrational rotation. Denjoy showed
that for the $C^2$ diffeomorphism of the circle without periodic
points such wandering interval cannot exist, and this semiconjugacy
is indeed a conjugacy. Analogue of Denjoy's theory also holds for a
$C^3$ unimodal map $f$ with a non-flat critical point (whose orders
are equal to the left and to the right) and negative Schwarzian
derivative, it was shown that $f$ admits no wandering intervals, and
it is conjugate to a quadratic map $f_\mu$ $(f_\mu=\mu x(1-x),
\mu\in (0,4))$ with some value of the parameter $\mu$
\cite{MR997312} if it has no periodic attractor. Secondly, the
nonexistence of wandering intervals is also relevant to the dynamics
of single dynamical system. A very remarkable result about
iterations of rational maps of the Riemann sphere proved by Sullivan
is that there are no wandering components in Fatou set, i.e., all of
the connected components of Fatou set of a rational map are
eventually periodic \cite{MR819553}. The Julia -Fatou-Sullivan
theory for the dynamics of rational map is also valid for $C^2$
multimodal maps with non-flat critical points whose orders are equal
to the left and to the right, it was shown that there are no
wandering intervals for such maps \cite{MR1161268}.

Besides the above results, the nonexistence of wandering intervals
was proved by a series remarkable works: by Schwartz for continuous
piecewise $C^1$ interval maps $f$ with $\log|Df|$ satisfying
Lipschitz condition in \cite{MR0155061}, by Guckenheimer for maps
which everywhere have negative Schwarzian derivative in
\cite{MR553966}, by Yoccoz for $C^\infty$ homeomorphisms of the
circle with non-flat critical points in \cite{MR741080}, by Blokh
and Lyubich for smooth interval maps with all critical points are
turning points in \cite{MR1036906} and \cite{MR1036905}, and by van
Strien and Vargas for general multimodal maps whose orders of each
critical point are equal from both sides  in \cite{MR2083467}.

However, the wandering intervals may exist for some maps. For the
continuous case, Denjoy's counterexample tell us that a $C^1$
differeomorphism on the circle may have a wandering interval
\cite{MR997312}, and there are $C^\infty$ maps on compact interval
with flat critical points which have wandering intervals
\cite{MR662469}. For the discontinuous case, there are wandering
intervals for a Lorenz map if it can be renormalized to be a gap map
with irrational rotation number, Berry and Mestel showed that
wandering intervals exist only in this case for the Lorenz map $f$
with $\log|Df|$ satisfying the Lipschitz condition in
\cite{MR1122907}. There exists affine interval exchange
transformations which have wandering intervals \cite{MR1488320}
\cite{MR2465670}.

From the above examples, we cannot expect that there are no
wandering intervals for discontinuous interval maps with critical
points like continuous case under some general conditions. On the
other hand, the known results of nonexistence of wandering
intervals for continuous multimodal maps always require the orders
of each critical point are the same to the left and to the right.
Blokh had asked whether wandering intervals can exist for unimodal
maps which are smooth except at their critical point and the
critical order is different to the left and to the right (this
question is quoted in \cite {MR1239171}, or \cite{MR1603750} for a
general case). In this paper, we consider the nonexistence of
wandering intervals of interval maps with some additional
conditions, and give a sufficient condition to answer Blokh's
question for general maps in some sense.

We now give the precise statement of our main result. An interval
$J$ is a {\it wandering interval} for a map $f:M \rightarrow M$ if
it satisfies:

(a) its forward iterates $J,f(J),...,f^n(J)$ are all disjoint for
all $n> 0$;

(b) $J$ is not contained in the basin of attraction of an
attracting periodic orbit;

(c) $f^n|_J$ is a homeomorphism for all $n> 0$.

Let $ \mathcal {A}$ denote the class of interval maps satisfying
conditions 1 and 2 listed below. Then we have the following

\vspace{0.3cm}

{\bf Theorem A. } {\it Wandering intervals can not exist for each
map in $\mathcal{A}$.} \vspace{0.3cm}

 1. Let
$M$ be a compact interval $[0, 1]$, $f: M \longrightarrow M$ be a
piecewise $C^1$ interval map with negative Schwarzian derivative.
This means that there exists a finite set $C$ such that $f$ is a
diffeomorphism on each component of $M \setminus C$, and admits a
continuous extension to the boundary so that the left and the right
limits $f(c{\pm})= \lim_{x \rightarrow c\pm} f(x)$ exist. We always
regard each $c \in C$ as two points: $c+$ and $c-$, the concrete
values depend on the corresponding one-side neighborhoods we
considered. We assume that each $c\in C$ has two one-side critical
orders $l(c\pm)\in [1, \infty)$, this means that
$$
 |Df(x)|\approx
{|x-c|}^{l( c\pm)-1}, \ \ |f(x)-f(c\pm )|\approx{|x-c|}^{l(
c\pm)},
$$ for $x$ in the corresponding one-side neighborhood of c, where we say $f\approx g$ if the ratio $f/g$ is bounded
above and below uniformly in its domain. When we use $l(c)$, it
may be either $l(c+)$ or $l(c-)$, and the concrete value can be
easily understood from the context. If $l(c) > 1$, we say that $c$
is a critical point, if $l(c)=1$, we say that $c$ is a bounded
derivative point. Note that $c$ may be a critical point on one
side and is a bounded derivative point. When there is no
possibility of confusion, each point $c \in C$ will be called a
critical point without distinguishing whether $c$ is really a
critical point with $l(c)>1$, or $c$ is a bounded derivative point
with $l(c)=1$.

We also assume that $f$ is with negative Schwarzian derivative
outside of $C$, i.e., $|f'|^{-\frac{1}{2}}$ is a convex function on
each  component of $M\setminus C$.

2. We suppose that $f$ satisfies the following summability
conditions along the critical orbits. The first summability
condition is
\begin{equation}\label{the first summability}
\sum_{n=1}^{\infty} { \Big{(} \frac {
{|f^n(c)-\tilde{c}|}^{l(\tilde{c})} }
{{|f^n(c)-\tilde{c}|}^{l(c)}|Df^n(f(c))|} \Big{)} }^{1/l(c)} <
\infty , \ \ \forall c \in C,
\end{equation}
 where $\tilde{c}$ is the critical point closest to $f^n(c)$, and
 $l(c)$ and $l(\tilde{c})$ depend on the corresponding one-side neighborhoods.
 The second summability condition is
 \begin{equation}\label{the second summability}
\sum_{n=1}^{\infty} \frac{1}{|Df^n(f(c))|^{1/l(c)}} < \infty, \ \ \forall c \in C.
\end{equation}

\remark According to the first summability condition, a map  $f \in
\mathcal {A}$ can not map its critical point to another critical
point, i.e., $C\cap \cup_{n\ge 1}f^n(C)=\emptyset$, and it is easy
to see that if all of the critical orders are equal, the first
summability condition is equivalent to the second summability
condition and $C\cap \cup_{n\ge 1}f^n(C)=\emptyset$. The above
similar summability conditions along the critical orbits have been
applied to show the existence of absolutely continuous invariant
measure for unimodal maps in \cite{MR1109621}, for multimodal maps
in \cite{MR1858488} and for interval maps possible with critical
points and singularities in \cite{abv} recently. The second
summability condition is similar to the Nowicki-van Strien condition
in \cite{MR1109621}.

\remark Simple examples satisfying the conditions 1 and 2 listed
above are the contracting Lorenz maps considered in \cite{MR1753089}
and \cite{MR1254985}, which were motivated by the study of the
return map of the Lorenz equations near classical parameter values,
see Figure 1. For Lorenz map $f$ with $\log|Df|$ satisfying the
Lipschitz condition, Berry and Mestel showed in \cite{MR1122907}
that wandering intervals exist if and only if it can be renormalized
to be a gap map with irrational rotation number. This, together with
Theorem A, implies that if a $C^1$ Lorenz map $f$ with $l(c)=1$
satisfying the summability condition
$\sum_{n=1}^{\infty}\frac{1}{|Df^n(f(c))|} < \infty$ and the
negative Schwarzian derivative condition, then it can not be
renormalized to be a gap map with irrational rotation number. On the
other hand, for $C^2$ Lorenz map $f$ whose critical orders are
greater than $1$, it is conjectured in an old version of
\cite{MR1836435} (arxiv: math/9610222v1) and in \cite{MR1715194}
that $f$ admits  wandering intervals if and only if it can be
renormalized to be a gap map with irrational rotation number, while
Theorem A indicates that wandering intervals can not exist for
Lorenz map $f$ satisfying $\sum_{n=1}^{\infty}
\frac{1}{|Df^n(f(c))|^{1/l(c)}} < \infty$.

\begin{figure}[hb]
 \centering
 \includegraphics[height=4cm]{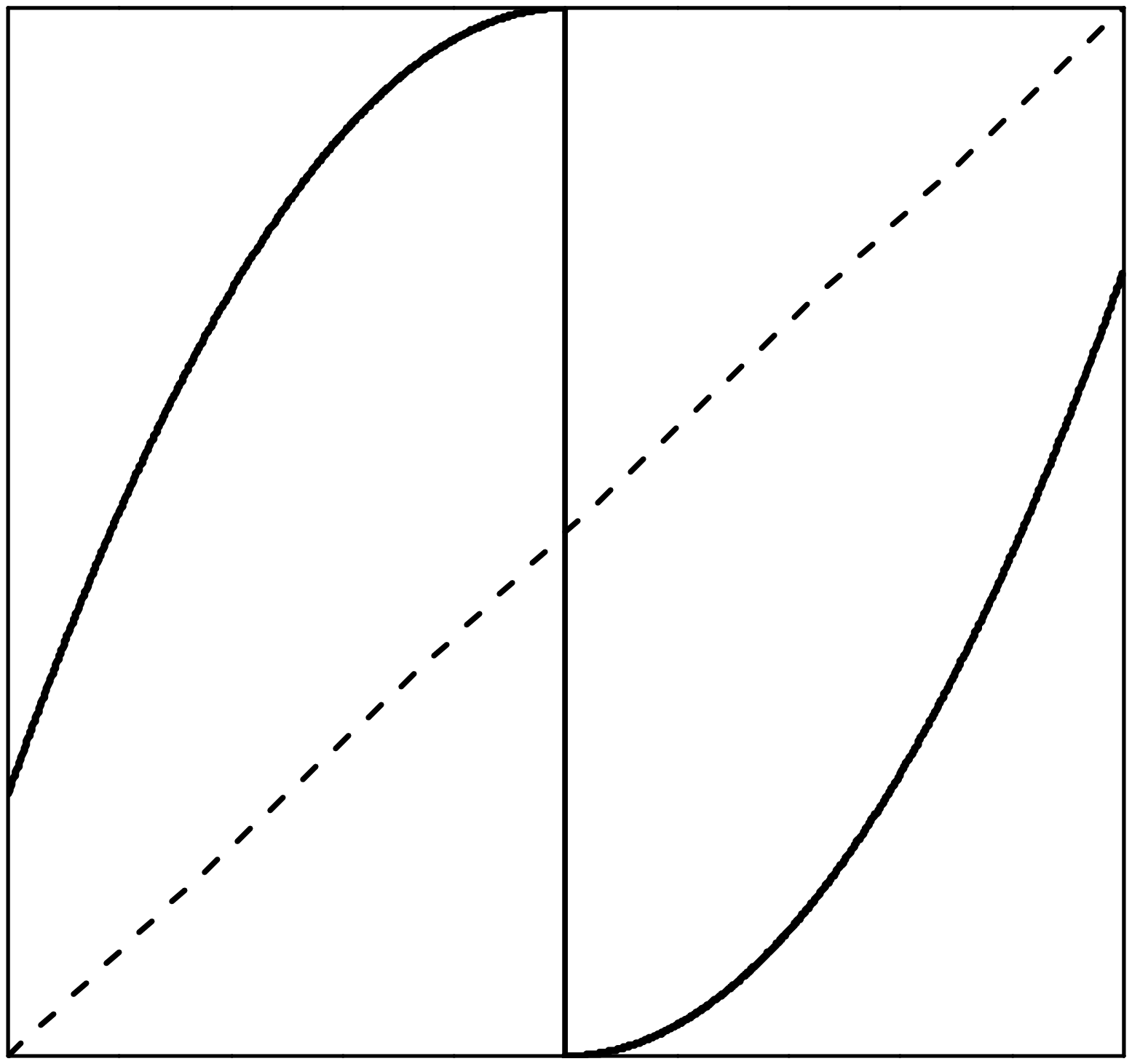}
   \includegraphics[height=4.3cm]{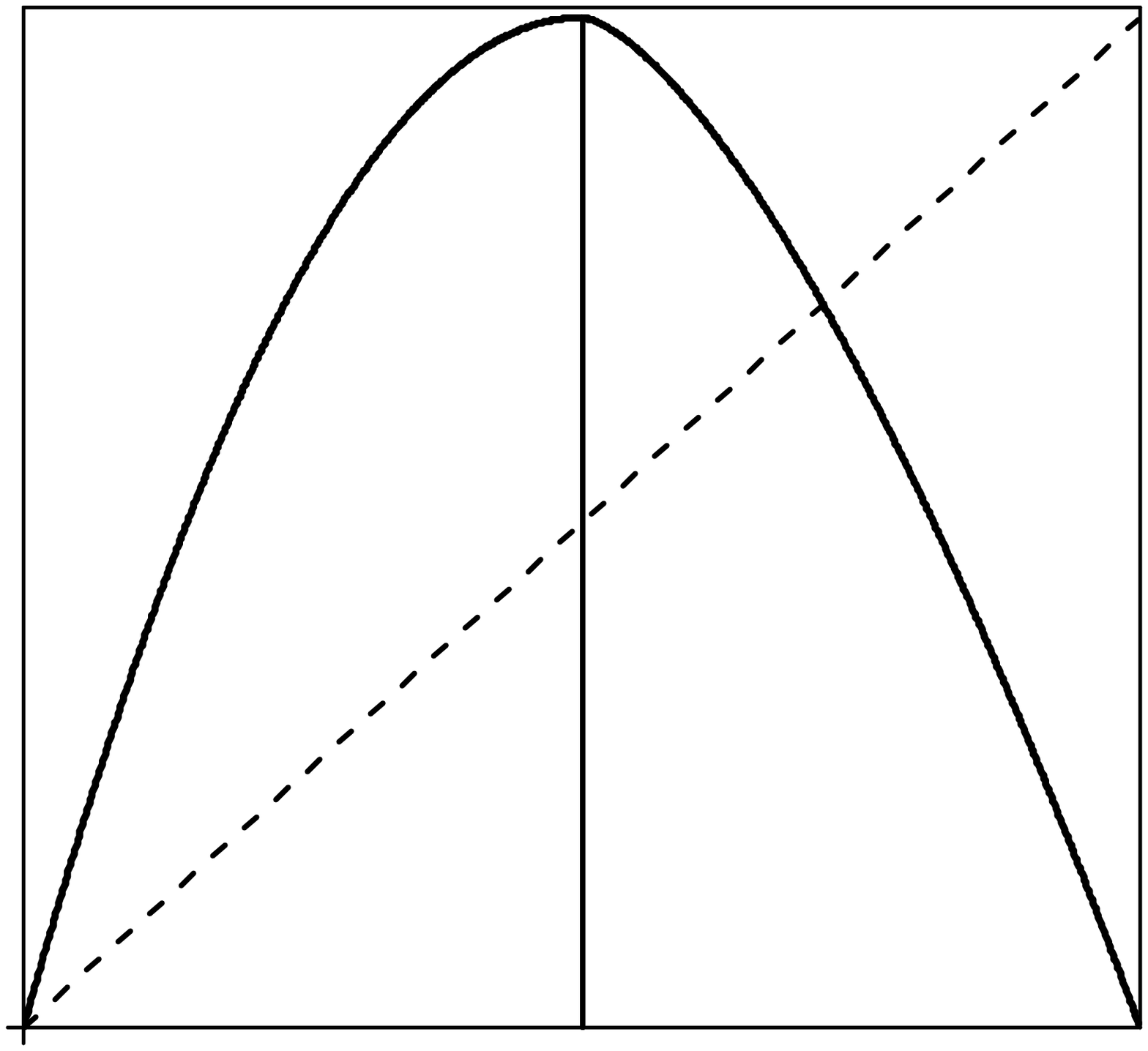}

 \caption{The contracting Lorenz map, and unimodal map with different critical orders.}
\end{figure}

In particular, for continuous maps in $\mathcal{A}$, we answer the
Blokh's question in some sense partially.

\begin{corollary}
Let $f$ be a $C^1$ multimodal map with finite critical points whose
orders may not equal to the left and to the right. If $f$ satisfies
the negative Schwarzian derivative condition and our summability
conditions, then $f$ admits no wandering intervals.
\end{corollary}

Furthermore, we can prove the following Theorem B.

\vspace{0.3cm}

 {\bf Theorem B. } {\it Let $f$ be a map in
 $\mathcal{A}$, then $f$ admits an absolutely continuous invariant probability measure
 (acip). Moreover, if $l_{\max}>1$,
  the density is in $L^p$ for all $1\leq p
  <l_{\max}/(l_{\max}-1)$, where $l_{\max}$ is the
  maximum of the orders of the critical points.}
\vspace{0.3cm}

We present some comments on Theorem B. At first, Theorem B improves
a result in \cite{MR1753089}.
 It was shown in \cite{MR1753089} that for
  contracting Lorenz maps $f$ satisfying the following conditions

  (1) Outside of the unique critical point $c$ ($l(c)>1$), $f$ is of $C^3$ classes and
  with negative Schwarzian derivative, and with equal critical
  orders from both sides;

  (2) (Increasing condition) $|Df^n(f(c\pm))|>{\lambda}^n$, for each $n\geq 1$ and some
  $\lambda>1$;

  (3) (Recurrence condition) $|f^{n-1}(f(c\pm))-c|>{\exp}^{-\alpha
  n}$ for some $\alpha$ small enough, and for all $n\geq 1$, \\
then $f$ admits an absolutely continuous invariant probability.
While Theorem B tell
  us that it suffices to assume the summability condition
  $$\sum_{n=1}^{\infty} \frac{1}{|Df^n(f(c))|^{1/l(c)}} < \infty
$$
and the above condition $(1)$  for the contracting Lorenz map.
Moreover, Theorem B provides the regularity of the density of the
acip.

Secondly, Theorem B is similar to a result in \cite{abv}. It was
proved in \cite{abv} that for any piecewise $C^2$ interval map $f$
possibly containing discontinuities and singularities ($0<l(c)<1$)
and satisfying an analytical condition (uniformly expanding away
from the critical points) and

(1) {\it bounded backward contraction (BBC) property}, there is a
constant $K>0$ such that for all $\delta_0>0$, one can find
$\delta\in (0, \delta_0)$ such that for a neighborhood $N_\delta$ (a
concrete definition of $N_\delta$ is stated in section 4) of the
critical points and each $x$,
$$|D^nf(x)|\geq K, \ \ \ for\ \ n=\min\{i\geq 0; f^i(x)\in N_\delta \},$$

(2) {\it summability condition along the critical orbit}
\begin{equation}\label{summability condition}
\sum_{n=1}^{\infty}\frac{-n\log|f^n(c)-C|}{|f^n(c)-C|{|Df^{n-1}(f(c))|}^{1/(2l(c)-1)}}
< \infty,\  \forall \ c \in C \ with \ l(c)>1,
\end{equation}
then $f$ admits only finite number of absolutely continuous
invariant (physical) probability measures. Note that the summability
condition (\ref{summability condition}) implies our first
summability condition (\ref{the first summability}) and the second
summability condition (\ref{the second summability}) in our argument
for $l(c)>1$, this is because for $n$ large enough,
\begin{equation*}
\begin{split}
&{\Big(\frac { {|f^n(c)-\tilde{c}|}^{l(\tilde{c})} }
{{|f^n(c)-\tilde{c}|}^{l(c)}|Df^n(f(c))|}\Big)}^{1/l(c)}
=\frac{{|f^n(c)-\tilde{c}|}^{l(\tilde{c})/l(c)-1}}{\big(|Df(f^n(c))|
\ |Df^{n-1}(f(c))|\big)^{1/l(c)}}
\\&\leq
\frac{{|f^n(c)-\tilde{c}|}^{1/l(c)-1}}{{|Df^{n-1}(f(c))|}^{1/l(c)}}\leq
\frac{-n\log|f^n(c)-C|}{|f^n(c)-C|{|Df^{n-1}(f(c))|}^{1/(2l(c)-1)}},
\end{split}
\end{equation*}
by the definition of orders of critical points and $l(c)>1$. As
was stated in \cite{abv}, in the special cases in which there are
no singularities ($0<l(c)<1$), the summability condition
(\ref{summability condition}) can not reduce to the summability
condition assumed in \cite{MR2013929}, while our conditions can.


Moreover, we can show BBC property under our conditions.

\vspace{0.3cm}
 {\bf Theorem C.} {\it Let $f \in \mathcal{A}$, then
$f$ admits BBC property.} \vspace{0.3cm}

As stated above, the BBC property is an assumption in \cite{abv}, is
also an assumption for the studying of the statistical properties
such as decay of correlations and the Central Limit Theorem of
interval maps in \cite{MR2285510}, and an important Lemma in
\cite{MR2013929}. BBC property is true  \cite{MR623533} for
symmetric unimodal maps with negative Schwarzian derivative.  It is
also proved in \cite{MR2013358} for  multimodal case with the same
critical orders of all critical points and
$$\lim_{n \to \infty} |Df^n(f(c))|=\infty, \ \ \forall c\in C.$$
It is interesting that the motivation of the proof the BBC
property in \cite{MR2013358} is to show the existence of the acip,
but in our argument we prove the BBC using the methods of the
proof of the existence of the acip. We emphasize that the proof of
the BBC property of mutlimodal case have to use the nonexistence
of wandering intervals for multimodal map [Theorem 1.2 in
\cite{MR2013358}] and the same critical order of all of critical
points. Using Theorem C, we can consider the decay of correlations
and CLT for interval maps possibly containing critical points and
discontinuities with negative Schwarzian derivative, under some
summability conditions on the growth of the derivative and
recurrence along critical orbits in our future work.

\section {Ideas and organization of the proof}
The proof of the nonexistence of wandering intervals of interval
maps will be achieved by contradiction, so we suppose that there is
a wandering interval $J$. According to Denjoy's original ideas, most
of the proofs (see, for example \cite{MR997312}, \cite{MR1161268})
of the nonexistence of wandering intervals contain two ingredients:
a topological one and an analytical one. The topological part is to
the understanding of the detailed dynamics of interval map, while
the analytical part is to estimate the distortion using the
topological results. Then, either $J$ would be attracted by a
periodic orbit or one can get a conclusion that contradicts to the
estimation of analytic aspect. Note that the condition that the
orders of each critical point from both sides are equal is necessary
in both ingredients of previous proofs.

Our proof is completely analytical, and has links to the method of
the proof of the existence of the acip. The method, as in
\cite{MR1109621}, is to  estimate of the Lebesgue measure of
$f^{-n}(A)$,  where $A$ is a small neighborhood of the critical
points. The proof of Theorems can be divided into three steps:

Step 1. Show that if $f\in \mathcal{A}$ then there exists a constant
$K_1>0$ so that  $|f^{-n}(B(c, \epsilon))| \leq K_1 \epsilon$, for
any $c\in C$ , $n \geq 0$ and small $\epsilon \geq 0$, where $B(c,
\epsilon)=(c-\epsilon, c+\epsilon)$. This property relates the
measure of preimages of a small neighborhood of the critical points
to the measure of neighborhood of the critical points.

Step 2. Show that if $f\in \mathcal{A}$ and satisfies the above
property in step 1, then there exists a constant $K_1>0$ so that
$|f^{-n}(A)| \leq K_1 {|A|}^{1/l_{\max}}$, where $A$ is any Borel
set. A more precise version of the above two steps is stated in
Section 3. The proof seems complicated, but it is almost the same
as the proof of the existence of the acip of unimodal map in
\cite{MR1109621}, multimodal map in \cite{MR1858488}.


Step 3. In Section 4, we present the proofs of Theorem A, Theorem B,
and Theorem C. Theorem A is a direct consequence of the property in
step 2. The proof of Theorem B is classical when $f$ is continuous.
In fact, if $f$ is continuous, the existence of acip follows from
Proposition 2 and the compactness of the space of probability
measure on $M$ under the weak star topology. This kind of argument
can not be applied to our case directly because $f$ may have
discontinuities. We obtain the existence of acip by checking the
uniformly integrability of the action of Frobenius-Perron operator
on the constant function ${\bf 1}$. To show Theorem C, the result in
step 1 indicates that if $f\in \mathcal{A}$ then there exists a
constant $K>0$ such that $|Df^n(x)|>K$ for $f^n(x)\in C$, because
the critical point will not be mapped into another critical point.
This property implies that the derivatives of the preimages of the
critical points are bounded away zero. Next we can find enough Koebe
spaces so that we can relate the derivatives of the preimages of the
critical point to the derivatives of the preimages of any point in
the neighborhood of the critical point by the one-side Koebe
principle. We refer \cite{MR1239171} for the proof of the Koebe
principle.

It is possible to improve these results. The negative Schwarzian
derivative condition may be omitted, however the strategy in
\cite{MR1815700} and \cite{MR2083467} made use of the nonexistence
wandering intervals, and the method in \cite{MR1161268} involves
detailed analysis of topological dynamics of intervals maps. Note
that the negative Schwarzian derivative condition rule out the
existence of the singularities ($0<l(c)<1$), once one can get rid
of this negative Schwarzian derivative condition, then the results
in this paper can generalized to the interval maps with
singularities and critical points. Secondly, general methods of
proving the existence of the acip of smooth maps may work
efficiently without using the result of the nonexistence of
wandering intervals, the quite general conditions are known to
guarantee the existence of the acip for smooth maps with a finite
number of critical points recently in \cite{MR2393079}, but they
used the results of the nonexistence of wandering intervals, too.
However, it was
 conjectured by Ara$\acute{u}$jo et al. in \cite{abv}  that it is not possible to obtain a general result about
the existence of acip in the presence of both critical points and
singularities by assuming conditions on the derivatives growth of
the critical points.

We denote by $K_l$ the constant from the orders of the critical
points, by $K_o$ from the Koebe principle, and denote $|J|$ as the
Lebesgue measure of $J$.

\section {backward contraction}

In this section, we will use the proof of the existence of acip of
unimodal maps and multimodal maps, see \cite{MR1109621} and
\cite{MR1858488}, and we refer the chapter 5 of  \cite{MR1239171}
for more details. We only give the main arguments, and do some
modifications with the proof of  multimodal maps in
\cite{MR1858488}.

 {\bf
Proposition 1.} {\it If $f \in \mathcal {A} $, then there exists a
constant $K_1
>0$ such that
$$
|f^{-n}(B(c, \epsilon))| \leq K_1 \epsilon
$$
for any $c\in C$ , $n \geq 0$ and small $\epsilon \geq 0$, where
$B(c, \epsilon)=(c-\epsilon, c+\epsilon)$. }

\begin{proof}
Denote $E_n(c, \epsilon) :=f^{-n}(B(c, \epsilon))$, we will divide
the components $I$ of $E_n(c, \epsilon)$ into three cases. Given
$\sigma >3\epsilon$, $\sigma$ is a constant to be fixed by the
summability conditions. Let $I \subset I' \subset I''$ be the
components of $E_n(c, \epsilon), E_n(c, 3\epsilon)$, and $E_n(c,
\sigma)$ respectively. We distinguish three cases:

\begin{enumerate}

\item  $I \in \mathcal{R}_n$--the regular case, if $f^n$ has no
critical point in $I''$.

\item $I \in \mathcal{S}_n$--the sliding case, if $f^n$ has a
critical point in $I''$ but not in $I'$.

\item $I \in \mathcal{T}_n$--the transport case, if $f^n$
contains a critical point in $I'$.
\end{enumerate}

{\bf The regular case:}  Suppose $I \in \mathcal{R}_n(c)$, we know
that $f^n(I'')$ contains a $1$-scaled neighborhood of $f^n(I)$, by
the Koebe principle there exists a constant $K_\mathcal{R}>0$ such
that
$$
\frac{2\sigma}{|I''|} =\frac{|f^n(I'')|}{|I''|} \leq K_\mathcal{R}
\frac{|f^n(I)|}{|I|} \leq K_\mathcal{R} \frac{2\epsilon}{|I|}.
$$

We choose $K_\mathcal{R}$ big enough so that this holds for all
critical point $c\in C$. Then we have
\begin{equation}\label{regular}
\sum_{I\in \mathcal{R}_n(c)} |I| \leq K_\mathcal{R}
\frac{\epsilon}{\sigma} \sum_{I \in \mathcal{R}_n(c)} |I''| \leq
K_\mathcal{R} \frac{\epsilon}{\sigma}.
\end{equation}

We shall show Proposition 1 by induction. The induction
hypothesis is
\begin{equation}\label{induction}
|E_k(c, \epsilon)| \leq  \frac{ 3K_\mathcal{R}}{\sigma} \epsilon
\end{equation}
for all $0<\epsilon \leq \frac{\sigma}{3}$, $c\in C$, and $k < n$.

The above inequality (\ref{induction}) is true for $n=2$, because for $\sigma$ sufficiently small, there are
only regular cases.

In what follows we shall prove the inequality (\ref{induction}) holds for  $k=n$ for the
sliding case and for the transport case.

Let $V(\sigma) =\min \{k \geq 1, |f^k(\tilde{c})-C|< \sigma \ \
for \ \  some \ \tilde{c}\in C\}$, observe that $V(\sigma)
\rightarrow \infty$ as $\sigma \rightarrow 0$ because $f$ can not map a critical point to another critical point by the summability  conditions.

{\bf The sliding case:} Let $I \in \mathcal{S}_n(c)$, $T^0 \supset
I$ be the maximal interval on which $f^n$ is a diffeomorphism.
Denote $T_0 =f^n(T^0)\supset f^n(I)=I_0 =B(c, \epsilon)$, $R_0$
and $A_0$ be the components of $T_0\setminus A_0$ and $|R_0| \geq
|A_0|$, denote $T^0:=[\alpha_0, \alpha_{-1}]$ with that
$f^n(\alpha_0) \in
\partial A_0$. Since $I \in \mathcal{S}_n(c) $, we have $|R_0|
\geq |A_0| \geq |I_0|$. We shall construct a sequence $n=n_0>
n_1> ...\geq 0$ and a nested sequence of intervals
\begin{equation}\label{nested}
T^s \supset ...\supset T^1
\supset T^0=T \end{equation}
as following:

\begin{enumerate}

\item Choose $0<n_1<n$ such that $f^{n_1}(\alpha_0)\in C$, let
$T^1:=[\alpha_1, \alpha_0]$ be the maximal interval containing
$T^0$ on which $f^{n_1}$ is a diffeomorphism. Assume that
$n_{i-1}$ and $T^{i-1}=[\alpha_{i-2}, \alpha_{i-1}]$ are defined,
define $n_i< n_{i-1}$ such that $f^{n_i}(\alpha_{i-1})\in C$ and
let $T^i:=[\alpha_i, \alpha_{i-1}]$ be the maximal interval which
contains $T^{i-1}$ and on which $f^{n_i}$ is a diffeomorphism.
Note that for $i \geq 1$, $T^i$ and $T^{i-1}$ have a precise
common boundary $\alpha_i$ and that $I \subset T^0 \subset
...\subset T^i$.

\item Let $k_i=n_i-n_{i+1}$, $I_{i+1}=f^{n_{i+1}}(I)$,
$T_{i+1}=f^{n_{i+1}}(T^{i+1})$, $R_{i+1}$ be the component of
$T_{i+1}\setminus I_{i+1}$ and contains a critical point in its
closure, and $A_i$ be another component. Let $L_{i+1} \subset
A_{i+1}$ be the interval adjacent to $I_{i+1}$ and satisfying
$f^{k_i}(L_{i+1})=R_i$, we conclude the following relationships
between the above intervals:
$$f^{k_i}(I_{i+1})= I_i, f^{k_i}(R_{i+1})=A_i, f^{k_i}(L_{i+1})=R_i.
$$

\item The above construction stops at $n_s=0$, when $|I_s|>|R_s|$
or when $|I_s| \leq |R_s| \leq |A_s|$.
\end{enumerate}
Note that $|I_i|\leq |R_i|$ and $|A_i|<|R_i|$ for $0 \leq i \leq
s-1$.

The following Lemma is to control  $|I_s|$.

\begin{lemma}(Proposition 1.3.1 in
\cite{MR1858488}) There exists $K_2>1$ such that
\begin{equation}
|I_s|\leq |f^n(I)|\prod_{i=0}^{s-1}K_2   {\Big{(} \frac{
{|f^{k_i}(c_{i+1})-\tilde{c}|}^{l(\tilde{c})- l(c_{i+1})}}
{|Df^{k_i}(f(c_{i+1}))|} \Big{)} }^{1/l(c_{i+1})},
\end{equation}

where $\tilde{c}$ is the critical point closest to $f^{k_i}(c_i)$,
$c_{i+1}$ is the critical point in $\partial R_{i+1}$.

\end{lemma}

\begin{proof}
This Lemma is almost the same as  Proposition 1.3.1 in
\cite{MR1858488}. Since the proof of the Lemma only uses the Koebe
principle, orders of the critical points and $|A_i|<|R_i|$ for $0
\leq i \leq s-1$, the result is also ture for $f\in \mathcal{A}$.
\end{proof}

Now we  compare $I$ with $I_s$. According to the stopping rules, we consider three cases:
\begin{enumerate}
\item If $n_s=0$,  then $I=I_s$ by the definition of $I_s$, it follows
that
\begin{equation} \label{1}
I\subset E_0(c_s, \frac{1}{2}|I_s|).
\end{equation}

\item If $|I_s|> |R_s|$, then we have
\begin{equation}\label{2}
I\subset f^{-n_s}(I_s) \subset  f^{-n_s}(I_s \cup R_s) \subset
E_{n_s}(c_s, 2|I_s|).
\end{equation}
\item If $|I_s|\leq|R_s|\leq |A_s|$, we use the 'sliding'
technology developed by Nowicki and van Strien in
\cite{MR1109621}. Let $J \subset T^s$ be an interval such that
$|J|=|I|$ and $G:=f^{n_s}(J)\subset I_s\cup R_s$ be adjacent to
$c_s$, since $|I_s|\leq|R_s|\leq |A_s|$, there exists constant
$K_o$ such that
$$
\frac{|G|}{|J|} \leq K_o \frac{|f^{n_s}(I)|}{|I|}
$$
by the one-side Koebe principle. Since  $|I|=|J|$, this gives that
$|G|\leq K_o|f^{n_s}(I)|$. Then
\begin{equation}\label{3}
J\subset E_{n_s}(c_s, K_o |I_s|), \ and \ |I| \leq |E_{n_s}(c_s,
K_o |I_s|)|.
\end{equation}
\end{enumerate}

\begin{lemma}\label{sigma}
One can choose $\sigma$ so small that the interval $T_i$ in (\ref{nested}) have
size less than $\sigma$ for $0 \le i<s$.
\end{lemma}

\begin{proof}
Because $v(\sigma)\rightarrow \infty$ as $\sigma\rightarrow 0$ and
$ { \Big{(} \frac { {|f^n(c)-\tilde{c}|}^{l(\tilde{c})} }
{{|f^n(c)-\tilde{c}|}^{l(c)}|Df^n(f(c))|} \Big{)}
}^{1/l(c)}\rightarrow 0 $ as $n\rightarrow \infty$, we can
choose $\sigma$ small enough so that for each $n\geq v(\sigma)$
one has
$${ \Big{(} \frac { {|f^n(c)-\tilde{c}|}^{l(\tilde{c})} }
{{|f^n(c)-\tilde{c}|}^{l(c)}|Df^n(f(c))|} \Big{)} }^{1/l(c)} \leq
\frac{1}{K},$$ where $K$ is a positive constant depending on $K_l,
K_o, l(c)$, and the finiteness of the number of critical points. Let
us show by induction that for small $\sigma$, $|R_i\cup I_i\cup
A_i|\leq \sigma$ for $i=0,...,s-1$. Assume that $s\geq 2$,
$|T_0|\leq \sigma$ follows from the definition. Observe that
\begin{equation}\label{11}|A_0|\leq |f^{k_0}(c_1)-c_0|\leq|A_0\cup I_0|\leq 2|A_0|.\end{equation}

We have
\begin{equation*}
\begin{split}
|T_1|&=|R_1\cup I_1 \cup A_1| \leq 3|R_1|\\
&\leq 3{K_l}^{1/l(c_1)}{|f(R_1)|}^{1/l(c_1)}      \ \   (oder\ of\ critical\ point)\\
&\leq 3{K_l}^{1/l(c_1)}{K_o}^{1/l(c_1)}
{\Big(\frac{|f(A_0)|}{|Df^{k_1}(f(c_1))|}\Big)}^{1/l(c_1)}  \ \   (one\ side\ Koebe\ principle)\\
&\leq 3{K_l}^{2/l(c_1)}{K_o}^{1/l(c_1)} {\Big(\frac{{|A_0|}^{l(c_0)}}{|Df^{k_1}(f(c_1))|}\Big)}^{1/l(c_1)}  \ \   (oder\ of\ critical\ point)\\
&=3{K_l}^{2/l(c_1)}{K_o}^{1/l(c_1)}|A_0| {\Big(\frac{{|A_0|}^{l(c_0)-l(c_1)}}{|Df^{k_1}(f(c_1))|}\Big)}^{1/l(c_1)} \\
&\leq 3{K_l}^{2/l(c_1)}{K_o}^{1/l(c_1)}|A_0| {\Big(\frac{{|f^{k_0}(c_1)-c_0|}^{l(c_0)-l(c_1)}}{|Df^{k_1}(f(c_1))|}\Big)}^{1/l(c_1)} \ \ \ \ \ \  (\ref{11})\\
&:=K |A_0|
{\Big(\frac{{|f^{k_0}(c_1)-c_0|}^{l(c_0)-l(c_1)}}{|Df^{k_1}(f(c_1))|}\Big)}^{1/l(c_1)}
\leq |A_0|<\sigma,
\end{split}
\end{equation*}

Similarly we can get for $2\leq i<s$,
\begin{equation*}
\begin{split}
 & |R_i\cup I_i \cup A_i| \leq 3|R_i| \leq
3{K_l}^{1/l(c_i)}{|f(R_i)|}^{1/l(c_i)}\\
&\leq 3{K_l}^{2/l(c_i)}{K_o}^{1/l(c_i)}
{\Big(\frac{|f(A_{i-1})|}{|Df^{k_{i-1}}(f(c_i))|}\Big)}^{1/l(c_i)}\\
&\leq 3{K_l}^{2/l(c_i)}{K_o}^{1/l(c_i)}
{\Big(\frac{{|A_{i-1}|}^{l(c_{i-1})}}{|Df^{k_{i-1}}(f(c_i))|}\Big)}^{1/l(c_i)}\\
& \leq 3{K_l}^{2/l(c_i)}{K_o}^{1/l(c_i)} |R_{i-1}|
{\Big(\frac{{|R_{i-1}|}^{l(c_{i-1})-l(c_i)}}{|Df^{k_{i-1}}(f(c_i))|}\Big)}^{1/l(c_i)}\\
& \leq 3{K_l}^{2/l(c_i)}{K_o}^{1/l(c_i)} |R_{i-1}|
{\Big(\frac{{|f^{k_{i-1}}(c_i)-c_{i-1}|}^{l(c_{i-1})-l(c_i)}}{|Df^{k_{i-1}}(f(c_i))|}\Big)}^{1/l(c_i)}\\
&:=K |R_{i-1}|{\Big(\frac{{|f^{k_{i-1}}(c_i)-c_{i-1}|}^{l(c_{i-1})-l(c_i)}}{|Df^{k_{i-1}}(f(c_i))|}\Big)}^{1/l(c_i)}\\
 &\leq |R_{i-1}|\leq \sigma,
\end{split}
\end{equation*}
where the last three inequalities are followed from the induction
assumption $(k_{i-1}>v(\sigma)$ and  $|T_{i-1}|\leq \sigma)$ and the
following relation
$$\frac{1}{3}|f^{k_{i-1}}(c_i)-c_{i-1}|\leq |R_{i-1}| \leq
|f^{k_{i-1}}(c_i)-c_{i-1}|.
$$
So we finish the proof of the Lemma.
\end{proof}

For every $s$-tuple $(k_0, k_1,...,k_{s-1})$, there are at most
${(2\sharp C)}^s$ intervals $I$ such that $f^{n_s}(I)$ slides to
the same interval $G$. Furthermore, by Lemma \ref{sigma}, it
follows that  $k_i \geq v(\sigma)$ for all $0\leq i <s$ from the
definition of $v(\sigma)$. Therefore, by (\ref{1}), (\ref{2}) and
 (\ref{3}), there exists a constant $K_S>4$ (depending on $K_o$)
such that
\begin{equation} \label{slinding}
\begin{split}
&\sum_{I\in \mathcal{S}_n(c)} |I| \leq  \sum_{c_s\in C} \sum_{k_j
\geq v(\sigma), \sum_jk_j=n-n_s\leq n } {(2\sharp C)}^s
|E_{n_s}(c_s, \frac{K_S}{2} |I_s|)|. \\
\end{split}
\end{equation}

{\bf The transport case:} Suppose $I\in \mathcal{T}_n(c)$, by the
definition of $\mathcal{T}_n(c)$, $f^n$ has at least one critical
point in $I'\supset I$. Let $k<n$ be the maximal integer such that
$f^k(I')$ contains a critical point $\tilde{c}$,  and denote the set
of such kind of intervals by $\mathcal{T}_n^k(\tilde{c}, c)$.
Clearly, $f^{n-k-1}$ maps $f^{k+1}(I')$ diffeomorphically into $B(c,
3\epsilon)$.

\begin{lemma} There exists $K_\mathcal{T} >0$ such that
\begin{equation}\label{transport}
\begin{split}
\sum_{I\in \mathcal{T}_n(c)} |I| &\leq \sum_{n-k\geq v(\sigma)}
\sum_{\mathcal{T}_n^k(\tilde{c}, c)} |I| \\
&\leq \sum_{\tilde{c}\in C}\sum_{n-k\geq
v(\sigma)}|E_{k}\big(\tilde{c},\ \ K_{\mathcal{T}}\epsilon
\frac{{|f^{n-k}(\tilde{c})-c|}^{(l(c)-l(\tilde{c}))/l(\tilde{c})}}{|Df^{n-k}(f(\tilde{c}))|^{1/l(\tilde{c})}}
\big{)}|.
\end{split}
\end{equation}
\end{lemma}

\begin{proof} Observe that $f^{n-k-1}$ maps $f^{k+1}(I')$ diffeomorphically into $B(c, 3\epsilon)$, $f(\tilde{c}) \in f^{k+1}(I')\subset [x,y]$,
where $f^{n-k-1}$ maps $[x, y]$ diffeomorphically onto $B(c,
3\epsilon)$, and $B(c, 3\epsilon)$ contains a $1$-scaled
neighborhood of $B(c, \epsilon)$. One can get from the one-side
Koebe principle that there exists $K_o$ such that
\begin{equation*}
\frac{|f^n(I)|}{|f^{k+1}(I)|}\geq K_o |Df^{n-k-1}(f(\tilde{c}))|,
\end{equation*}
i.e.,
$$
|f^{k+1}(I)| \leq
K_o\Big(\frac{|f^n(I)|}{|Df^{n-k-1}(f(\tilde{c}))|}\Big).
$$ From the
orders of the critical points, it follows that
$$
|f^{k}(I)| \leq K_l
{K_o}^{1/l(\tilde{c})}{\Big(\frac{|f^n(I)|}{|Df^{n-k-1}(f(\tilde{c}))|}\Big)}^{1/l(\tilde{c})}.
$$
Since $f^{n-k}(\tilde{c}) \in B(c, 3\epsilon)$, the Chain Rules
and orders of the critical points indicate that
\begin{equation*}
\begin{split}
|Df^{n-k}(f(\tilde{c}))|  & =|Df^{n-k-1}(f(\tilde{c}))|
|Df(f^{n-k-1}(f(\tilde{c})))| \\
& \leq K_l|Df^{n-k-1}(f(\tilde{c}))|
{|f^{n-k-1}(f(\tilde{c}))-c|}^{l(c)-1}.
\end{split}
\end{equation*}
Therefore, there exists a constant $K_3$(depending $K_o$ and
$K_l$) such that
\begin{equation}\label{transport2}
\begin{split}
|f^k(I)| &\leq K_l
{K_o}^{1/l(\tilde{c})}{\Big(\frac{K_l|f^n(I)|}{|Df^{n-k}(f(\tilde{c}))|{|f^{n-k}(\tilde{c})-c|}^{1-l(c)}}\Big)}^{1/l(\tilde{c})}
\\
&=K_l
{K_o}^{1/l(\tilde{c})}{\Big(\frac{K_l|f^n(I)|{|f^{n-k}(\tilde{c})-c|}^{l(c)-l(\tilde{c})}{|f^{n-k}(\tilde{c})-c|}^{l(\tilde{c})-1}}{|Df^{n-k}(f(\tilde{c}))|}\Big)}^{1/l(\tilde{c})}
\\&\leq K_3\epsilon
{\Big(\frac{{|f^{n-k}(\tilde{c})-c|}^{l(c)-l(\tilde{c})}}{|Df^{n-k}(f(\tilde{c}))|}\Big)}^{1/l(\tilde{c})}.
\end{split}
\end{equation}
Here the last inequality follows from the following relationship
$$
|f^n(I)|\leq 2\epsilon, \ l(\tilde{c}) \ge 1 \ \ and \
|f^{n-k}(\tilde{c})-c|\leq 3\epsilon.
$$
 Since
the number of critical points is finite and inequality
(\ref{transport2}), we know that there exists a constant
$K_{\mathcal{T}}$ such that for all critical point $\tilde{c} \in
C$,
\begin{equation*} \begin{split} I & \subset  f^{-k}(f^k(I))\\
 & \subset E_{k}\Big{(}\tilde{c},\ \
K_{\mathcal{T}}\epsilon
\frac{{|f^{n-k}(\tilde{c})-c|}^{l(c)/l(\tilde{c})-1}}{|Df^{n-k}(f(\tilde{c}))|^{1/l(\tilde{c})}}
\Big{)}.
\end{split}\end{equation*}
Since the definition of $v(\sigma)$ implies $n-k>
v(\sigma)$, summing over all such $I$ gives Lemma 2.

\end{proof}

We proceed to show Proposition 1, by  Lemma 4.9 in
\cite{MR1239171} and the summability conditions, one can choose
$\sigma$ so small that for $n>1$,
\begin{equation}\label{small sigma 1}
\sum_{c_s\in C} \sum_{k_j \geq v(\sigma), \sum_jk_j=n-n_s\leq n }
3K_S\prod_{i=0}^{s-1} (2\sharp C) K_2{\Big{(} \frac{
{|f^{k_i}(c_{i+1})-\tilde{c}|}^{l(\tilde{c})- l(c_{i+1})}}
{|Df^{k_i}(f(c_{i+1}))|} \Big{)} }^{1/l(c_{i+1})}\leq 1,
\end{equation}
and
\begin{equation} \label{small sigma 2}
\sum_{ \tilde{c}\in C}\sum_{n-k\geq v(\sigma)} 3K_{\mathcal{T}}
{\Big(\frac{{|f^{n-k}(\tilde{c})-c|}^{l(c)-l(\tilde{c})}}{|Df^{n-k}(f(\tilde{c}))|}\Big)^{1/l(\tilde{c})}}\leq
1.
\end{equation}
 So, for $n-k\geq v(\sigma)$, we have the following inequalities,
\begin{equation}\label{IS}
\frac{K_S}{2}|I_s|\leq \frac{\sigma}{3},
\end{equation}
and
\begin{equation}\label{IT} K_{\mathcal{T}} \epsilon
\Big(\frac{{|f^{n-k}(\tilde{c})-c|}^{l(c)-l(\tilde{c})}}{|Df^{n-k}(f(\tilde{c}))|}\Big)^{1/l(\tilde{c})}
\leq \frac{\sigma}{3}.
\end{equation}
 With the notations from the beginning of the proof, we have
$$
E_n(c, \epsilon)=\bigcup_{I\in \mathcal{R}_n}I \cup \bigcup_{I\in
\mathcal{S}_n}I \cup \bigcup_{I\in \mathcal{T}_n}I,
$$
and
$$
|E_n(c, \epsilon)|\leq \sum_{I\in \mathcal{R}_n}|I| + \sum_{I\in
\mathcal{S}_n}|I| + \sum_{I\in \mathcal{T}_n}|I|.
$$
Therefore, by (\ref{regular}), (\ref{slinding}),
(\ref{transport}), (\ref{IS}), and (\ref{IT}),  and using the
induction hypothesis we have
\begin{equation*}
\begin{split} & |E_n(c, \epsilon)| \leq  \frac{ K_\mathcal{R}}{\sigma}
\epsilon \ \ + \\
&\sum_{c_s\in C} \sum_{k_j \geq v(\sigma), \sum_jk_j=n-n_s\leq n }
K_S \frac{3K_R}{\sigma} \epsilon \prod_{i=0}^{s-1} K_2(2\sharp C){
\Big{(} \frac { {|f^{k_i}(c)-\tilde{c}|}^{l(\tilde{c})- l(c)} }
{|Df^{k_i}(f(c_i))|} \Big{)} }^{1/l(c)}  \\
&\ +\ \sum_{\tilde{c }\in C}\sum_{n-k \geq v({\sigma})}
K_{\mathcal{T}} \frac{3K_R}{\sigma} \epsilon
\frac{{|f^{n-k}(\tilde{c})-c|}^{l(c)/l(\tilde{c})-1}}{|Df^{n-k}(f(\tilde{c}))|^{1/l(\tilde{c})}}.
\end{split}
\end{equation*}
Then by the choice of $\sigma$, and by (\ref{small sigma 1}), and
(\ref{small sigma 2}), we obtain for $\epsilon <\frac{\sigma}{3}$,
$$
|E_n(c, \epsilon)| \leq \frac{ 3K_\mathcal{R}}{\sigma} \epsilon,
$$
which completes the proof.
\end{proof}

\vspace{0.3cm}

 {\bf Proposition 2:} Let $A$ be any measurable set,
then there is a constant $K_5>0$ such that

$$|f^{-n}(A)| \leq K_5 {|A|}^{1/l_{\max}}.$$

\begin{proof} Since Proposition 1 says that there is a control of the measure
of the preimages of a small neighborhood containing a critical
point, the proof can be divided into two parts. The first part is
to bound the measure of the preimages of a small intervals ``at
the end of branches" by the measure of the preimages of a small
neighborhood containing a critical point via the following Lemma.
\begin{lemma} \label{branch interval} (Lemma 1.2.1 in \cite{MR1858488})
Let $f \in \mathcal{A}$, there exists $K_4>0$ such that any
interval $I$ for which $f^n|_I$ is monotone and continuous,
$m(f^n(I))\leq \epsilon$ and one of the boundary points of $I$ is
a critical point of $f^n$, we have
 $$I\subset E_i(c,
 K_4\Big(\frac{{\epsilon}}{{|Df^{n-i-1}(f(c))|}}\Big)^{1/l(c)})$$
 for some $0\leq i \leq n$($i=n$ put $Df^{n-i-1}(f(c))=1$ ) and $c\in C$.
\end{lemma}

 The second part is to relate the measure of the preimages of any small Borel
 sets to the measure of the preimages of the intervals
``at the end of the branches" satisfying Lemma 3 by using the
Minimum principle. We only consider the case $|A|=\epsilon >0$, let
$J=(a, b)$ be a branch of  $f^n$, denote $\mathcal{J}$ by the all of
the branches of $f^n$,  choose $J_-=(a, d_1)$, and $J_+=(d_2, b)$
satisfies $f^n(J_\pm)=\epsilon$. By the Minimum principle (Lemma 4.2
in \cite{MR1239171}), it follows that
$$
|f^{-n}(A)\cap J|\leq K_o|J_- \cup J_+|.
$$
On the other hand, Lemma  \ref{branch interval} indicates that

$$
J_\pm \subset E_{i^{\pm}}\Big(c^{\pm},
 K_4\frac{1}{|Df^{n-i^{\pm}-1}(f(c^{\pm}))|^{1/l(c^{\pm})}}
 {\epsilon}^{1/l(c^{\pm})}\Big),
$$
where $c^+=f^{i^+}(a)$, and $c^-=f^{i^-}(b)$. Thus, by Proposition
1, it follows that
\begin{equation*}
\begin{split}
|f^{-n}(A)|&\leq \sum_{J \in \mathcal{J}}|f^{-n}(A)\cap J|
\leq \sum_{J \in \mathcal{J}}K_o|J_- \cup J_+|\\
&\leq  \sum_{J \in \mathcal{J}} K_o| E_{i^{\pm}}(c^{\pm},
 K_4{\Big(\frac{1}{|Df^{n-i^{\pm}-1}(f(c^{\pm}))|}\Big)}^\frac{1}{l(c^{\pm})} {\epsilon}^{1/l(c^{\pm})})|\\
 &\leq 2 \sum_{c\in C}\sum_{i=0}^{n-1}K_oK_1K_4
 {\Big(\frac{1}{|Df^{n-i-1}(f(c))|}\Big)}^{1/l(c)}{\epsilon}^{1/l(c)}\\
&\leq K_5 {\epsilon}^{1/l_{\max}},
\end{split}
\end{equation*}
 where $l_{\max}$ is the maximum of the orders of the critical
points, the last inequality follows from the second summability
condition (\ref{the second summability}). This implies Proposition
2.

\end{proof}

\section{proof of theorems }
{\it The proof of theorem A.} We argue by contradiction. Suppose
there exists a wandering interval $J$ for $f \in \mathcal {A}$.
Then we have $\sum_{n=0}^{\infty}|f^n(J)|< \infty$. Since $J,
f(J), \cdots, f^n(J), \cdots$
 are disjoint, $|f^n(J)|\rightarrow 0$ as $n \rightarrow
+\infty$. By Proposition 2, we get
$$|J| \leq |f^{-n}(f^{n}(J))| \leq K_5
{|f^{n}(J)|}^{1/l_{\max}}, \ \ \forall \ n>0.$$ Therefore, let $n
\rightarrow \infty$, the above inequality implies that $|J|=0$,
which contradicts to $J$ is a wandering interval. Theorem A is
proved. $\hfill \Box$


{\it The proof of theorem B.} Let $f\in \mathcal{A}$, denote $P_f$
as the Fronbenius-Perron operator induced by $f$. We consider the
sequence of functions $\{P_f^n {\bf 1}\}_{n=0}^\infty$, where
${\bf 1}$ is the constant function on the unit interval $M$. By
Proposition 2, it follows that $\{P_f^n{\bf 1}\}_{n=0}^\infty$ is
uniformly integrable, i.e., for every $\epsilon>0$, there is
$\delta>0$ such that
$$\int_AP_f^n{\bf 1}dx =\int_{f^{-n}(A)}{\bf 1}dx\leq K_5
{|A|}^{1/l_{\max}}\leq \epsilon,$$ if $m(A)\leq \delta$, $n>0$. On
the other hand,  since the $L^1$ norm of $P_f^n{\bf 1}$ is equal
to $1$, by a result from functional analysis \cite{DS},
$\{P_f^n{\bf 1}\}_{n=0}^\infty$ is weakly precompact in $L^1$. The
same applies to the sequence
$\{g_n:=\frac{1}{n}\sum_{j=0}^{n-1}P_f^j{\bf 1}\}$, there is a
subsequence $g_{n_k}$ that converges weakly to $g_*$,
$P_fg_*=g_*$. By the abstract ergodic theorem of Kakutani and
Yosida, $g_{n}$ converges strongly to $g_*$, and $g_*$ is an
invariant density of $f$ \cite{LM}. As a result, the measure
defined by $$\mu(A)=\int_A g_* (x) dx, \ \ \  A\  is\  a\  Borel\
set$$ is an invariant probability measure of $f$.

By Proposition 2,  we have $\mu(A) \leq K_5 {m(A)}^{1/l_{\max}}$.
If $l_{\max}>1$, then the probability density  $g_*$ is a $L^p$
function for $1\leq p < l_{\max}/(l_{\max}-1)$. $\hfill \Box$

\vspace{0.3cm}

 {\it The proof of theorem C.}  For a small $\delta$, denote by
 $N_{\delta}(c):=\{x; |f(x)-f(c)|<\delta \}$ a small neighborhood containing $c$, and
$N_{\delta}:=\cup_{c\in C} N_{\delta}(c)$. Assume that $x,f(x),...,$
$f^{n-1}(x) \notin N_\delta$, and $f^n(x)\in N_\delta$, let $T_n(x)$
be the maximal interval containing $x$ on which $f^n$ is a
diffeomorphism.

 {\bf Claim:} If $f^n(x)\in  N_{\delta}(c)$,
then $B(f^n(x), |N_{\delta}(c)|)\subset f^n(T_n(x)).$

Indeed, let $T_n(x)=(a, b)$, assume without loss of generality that
$|f^n(x)-f^n(a)|\leq |f^n(b)-f^n(x)|$, we then assume
$|f^n(x)-f^n(a)|\leq |N_{\delta}(c)|$ by contradiction. So there is
a critical point $c_j \in C$ and $0\leq n_1<n$ such that
$f^{n_1}(a)=c_j$. The orders of the critical point imply
\begin{equation}\label{nonflat}
|f^{n-n_1}(c_j)-c|\leq 2K_l {\delta}^{1/l(c)}.
\end{equation}
Observe that $x,f(x),...,f^{n-1}(x) \notin N_\delta$ and
$f^{n_1}(x)$ lies in a small neighborhood of $c_j$, then we have
$|f^{n_1+1}(x)- f(c_j)| \geq \delta$. Because $f^{n-n_1-1}$ is a
diffeomorphism on $(f^{n_1+1}(x), f(c_j))$, by the mean value
theorem, there exists $z\in (f^{n_1+1}(x), f(c_j))$ such that
$$
|Df^{n-n_1-1}(z)|= \frac{|f^{n-n_1}(c_j)-f^n(x)|}{|f^{n_1+1}(x)-
f(c_j)|} \leq \frac{ K_l{\delta}^{1/l(c)}}{\delta}=
K_l{\delta}^{-1+1/l(c)}.
$$
On the other hand, since $|f^n(x)-f^n(a)|\leq |f^n(b)-f^n(x)|$ and
$f^{n-n_1-1}(z)\in (f^n(a), f^n(x))$, the one-side Koebe principle
indicates
$$ |Df^{n-n_1-1}(z)| \geq K_o |Df^{n-n_1-1}(f(c_j))|.$$
Therefore, we obtain that
\begin{equation*}
\begin{split}
|Df^{n-n_1}(f(c_j))|&=|Df^{n-n_1-1}(f(c_j))| |Df(f^{n-n_1}(c_j))|
\\&\leq \frac{K_l}{K_o}|Df^{n-n_1-1}(z)|{|f^{n-n_1}(c_j)-c|}^{l(c)-1}
\\& \leq \frac{K_l}{K_o}K_l
{\delta}^{-1+1/l(c)}{(2K_l{\delta}^{1/l(c)})}^{l(c)-1}=\frac{K_l}{K_o}2^{l(c)-1}{K_l}^{l(c)}.
\end{split}
\end{equation*}

Since $|Df^n(f(c_j))|\rightarrow \infty$ as $n \rightarrow \infty$,
the above estimation implies that $n-n_1$ is bounded above by some
integer $n_0$ which does not depend on $\delta$. On the other hand,
the (\ref{nonflat}) indicates that these exists $\delta_0>0$ such
that for any $\delta <\delta_0$, one gets $n-n_1>n_0$. So we obtain
a contradiction. The Claim is true.

Next we assume  $y\in T_n(x)$ such that $f^n(y)=c$, then there
exists $K_1>0$ such that $|Df^n(y)|>K_1$ by Proposition 1. On the
other hand, $f^n(x)\in N_{\delta}(c)$ and the above claim give
$$
|f^n(x)-c|\leq |N_\delta(c)| \ \ \ and \ \ d(f^n(x), f^n(\partial
T_n(x)))\geq |N_{\delta}(c)|.$$

Therefore, by the one-side Koebe principle, it follows that there
exists $K_6>0$ ($K_6$ is not depending on $\delta$) such that
$$
|Df^n(x)|\geq K_o|Df^n(y)| \geq K_o K_1\geq K_6.$$ $\hfill \Box$

\remark We only need the $|Df^n(f(c_j))|\rightarrow \infty$  as $n
\rightarrow \infty$ and result in Proposition 1 in the proof of
Theorem C, a stronger version of BBC property of multimodal maps
with equal critical orders from two sides under a stronger
increasing condition can be found in \cite{Cedervall}.


 \end{document}